\documentclass[11pt]{article}
\usepackage{amsxtra}
\usepackage{mathrsfs}
\usepackage{yfonts}

\usepackage{mathtools}
\usepackage{amssymb}
\usepackage{amsthm}
\usepackage{mathdots}
\usepackage[]{graphicx}
\usepackage{wrapfig}
\usepackage{float}
\usepackage{verbatim}
\usepackage[margin=1in]{geometry}
\usepackage[utf8]{inputenc}
\usepackage[T1]{fontenc}
\newtheorem{theorem}{Theorem}
\newtheorem{proposition}[theorem]{Proposition}
\newtheorem{rem}[theorem]{Remark}
\newtheorem{ex}[theorem]{Example}
\usepackage{color}
\newtheorem{lemma}[theorem]{Lemma}

\newcommand{\Diff}{\mathscr{D}}
\newcommand{\Diffmu}{\Diff_{\mu}}

\newcommand{\llangle}{\langle\!\langle}
\newcommand{\rrangle}{\rangle\!\rangle}

\newcommand{\transpose}{\text{T}}

\newcommand{\grad}{\nabla}
\newcommand{\Laplacian}{\Delta}

\DeclareMathOperator{\diver}{div}

\title{The exponential map of the group of area-preserving diffeomorphisms of a surface with boundary}
\author{James Benn, Gerard Misio{\l}ek and Stephen C. Preston}

\date{\today}
\begin{document}

\maketitle
\numberwithin{equation}{section}

\begin{abstract}
We prove that the Riemannian exponential map of the right-invariant $L^2$ metric on the group of
volume-preserving diffeomorphisms of a two-dimensional manifold with a nonempty boundary
is a nonlinear Fredholm map of index zero.
\end{abstract}

\section{Introduction}
\label{sec:Intro}

\hskip 0.5cm
Consider a compact $n$-dimensional manifold $M$ with a smooth boundary $\partial M$
equipped with a Riemannian metric.
Let $\mathscr{D}_{\mu}^{s}$ be the volumorphism group; that is, the group of diffeomorphisms of $M$ which preserve
the Riemannian volume form $\mu$ and are of Sobolev class $H^s$.
It is well-known that if $s> n/2+1$, then $\mathscr{D}_{\mu}^{s}$ is a submanifold of
the infinite dimensional Hilbert manifold $\mathscr{D}^s$ of all $H^s$ diffeomorphisms of $M$.
Its tangent space $T_\eta\mathscr{D}^s_\mu$ consists of $H^{s}$ sections $X$ of
the pull-back bundle $\eta^{*}TM$ whose right-translations $X{\circ}\eta^{-1}$ to the identity element
are the divergence-free vector fields on $M$ that are parallel to the boundary $\partial M$. The $L^{2}$ inner product for vector fields
\begin{equation}
\langle u, v \rangle _{L^{2}}
=
\int_{M} \langle u(x), v(x) \rangle \, d\mu(x)
\qquad\quad
u,v\in T_{e}\mathscr{D}_{\mu}^{s}
\label{eq:L2}
\end{equation}
defines a right-invariant metric on $\mathscr{D}^s$ and hence also on $\mathscr{D}_\mu^s$
with associated Levi-Civita connections.
The curvature tensor $\mathcal{R}$ of this metric on $\mathscr{D}_\mu^s$ is a bounded trilinear operator
on each tangent space and is invariant with respect to right translations by $\mathscr{D}_{\mu}^{s}$.
Our main references for the basic facts about $\mathscr{D}^s_\mu$ and its $L^2$ geometry
are the papers \cite{Ebin70}, \cite{Misiolek92}, \cite{MP10} and the monograph \cite{Arnold98}.

\hskip 0.5cm
Arnold, in his pioneering paper \cite{Arnold66}, reinterpreted the hydrodynamics of
an ideal fluid filling $M$ in terms of the Riemannian geometry of the volumorphism group of $M$
equipped with the $L^{2}$ metric describing the fluid's kinetic energy.
He showed that a curve $\eta(t)$ is a geodesic of the $L^{2}$ metric on $\mathscr{D}^s_\mu$
starting from the identity element $e$ in the direction $v_{0}$ if and only if
the time dependent vector field $v=\dot{\eta}\circ\eta^{-1}$ on $M$ solves
the incompressible Euler equations
\begin{align} \label{eq:euler}
&\partial_{t} v + \nabla_{v}v = - \mathrm{grad}\, p
\nonumber \\
&\mathrm{div}\, v=0
\\  \nonumber
&\langle v, \nu \rangle = 0 \;\, \mathrm{on} \;\, \partial M
\end{align}
with the initial condition
\begin{align} \label{eq:eulerIC}
v(0) = v_{0}
\end{align}
where $p$ is the pressure function, $\nabla$ denotes the covariant derivative on $M$
and $\nu$ is the outward pointing normal to the boundary $\partial M$.

\hskip 0.5cm
It turns out that there is a technical advantage in rewriting the Euler equations this way;
Ebin and Marsden \cite{Ebin70} showed that the Cauchy problem for the corresponding geodesic equation
in $\mathscr{D}^s_\mu$ can be solved uniquely on short time intervals by a standard Banach-Picard
iteration argument.
In particular, its solutions depend smoothly on the data, and as a result one can define
(at least for small $t$) a smooth exponential map
$$
\exp_{e}:T_{e}\mathscr{D}_{\mu}^{s} \rightarrow \mathscr{D}_{\mu}^{s},
\qquad
\exp_e tv_0 = \eta(t),
$$
where $\eta(t)$ is the unique geodesic of \eqref{eq:L2} issuing from the identity with initial velocity
$v_0 \in T_{e}\mathscr{D}_{\mu}^{s}$.
The exponential map is a local diffeomorphism from an open set around zero in
 $T_{e}\mathscr{D}_{\mu}^{s}$ onto a neighborhood of
the identity in $\mathscr{D}_{\mu}^{s}$.
This follows from the inverse function theorem and the fact that the derivative of $\exp_e$ at time $t=0$ is the identity map.
Furthermore, if $n=2$ then by the classical result of Wolibner \cite{Wolibner33}
the exponential map can be extended to the whole tangent space $T_e\mathscr{D}_\mu^s$, which is interpreted
as geodesic completeness of the volumorphism group with respect to the $L^2$ metric.
These results continue to hold in the case when the underlying Riemannian manifold $M$ is noncompact
with bounded geometry and compact boundary,
provided that we restrict to diffeomorphisms that differ from the identity only on a compact subset of $M$
or, more generally, to those that fall off rapidly to the identity at each end of $M$,
cf. e.g., \cite{Cantor}, \cite{EichSchm}, or \cite{Michor}.

\hskip 0.5cm
The structure and distribution of singularities of the exponential map of \eqref{eq:L2}
has been of considerable interest ever since the problem of conjugate points in $\mathscr{D}_\mu^s$
was raised by Arnold in \cite{Arnold66}.
The first examples of conjugate points were constructed in \cite{Misiolek92} and \cite{Misiolek96}
in the case when $M$ is a sphere with the round metric or the flat 2-torus.
Further examples can be found in \cite{Shnirelman94}, \cite{Preston06}, \cite{Preston08}, \cite{Benn1}
and \cite{Benn2}.
In \cite{EMP06} it was proved that the $L^2$ exponential map is a non-linear Fredholm map of index zero
whenever $M$ is a compact manifold of dimension $2$ without boundary and moreover that
the Fredholm property fails for a steady rotation of the solid torus in $\mathbb{R}^3$.
More pathological counterexamples were constructed in \cite{Preston06} using curl eigenfields
on the sphere $\mathbb{S}^3$ and more recently in \cite{PrWash} in the case of certain axisymmetric flows
in $\mathbb{R}^3$.
Furthermore, Shnirelman \cite{Shnirelman05} proved that when $M$ is the flat 2-torus the exponential map on
$\mathscr{D}_\mu^s$ is a Fredholm quasiruled map. In \cite{MP10} the authors
showed that the failure of the Fredholm property in the case of three-dimensional manifolds is ``borderline,''
in the sense that the exponential maps of Sobolev $H^r$ metrics are necessarily Fredholm whenever $r>0$.

\hskip 0.5cm
An outstanding problem left unresolved in these papers concerns the case when
a two-dimensional manifold $M$ has a nonempty boundary $\partial M$.
The methods employed in \cite{EMP06} 
allowed only for a much weaker result,
namely, that the derivative of the exponential map along a geodesic in $\mathscr{D}^s_\mu$ can be extended to
a linear Fredholm operator defined on the $L^2$ completions of the tangent spaces to the volumorphism group.
The question of whether the behavior is genuinely different in case of a boundary has been raised in light of
recent work  where phenomena have been discovered that seem to rely heavily on the
presence of the boundary (such as double-exponential growth of the vorticity field in 2D~\cite{KisSve}
and numerically-observed blowup in 3D~\cite{LuoHou}).

\hskip 0.5cm
The main goal of the present paper is to establish the strong $H^{s}$ Fredholmness property of the exponential map
for incompressible 2D fluids in the presence of boundaries.
For notational simplicity and clarity of exposition we will consider the simplest case of
periodic flows in the upper half-plane and work in a single chart.
The general case of bounded domains in $\mathbb{R}^2$ can be treated
in the standard way by choosing a suitable open cover of the boundary $\partial M$
together with a subordinate smooth partition of unity and applying the result for the half-plane.
Our main result in this paper is thus the following
\begin{theorem} \label{thm:Fred-b}
Let $M=S^1 \times [0,\infty)$ be the periodic upper half-plane with boundary $\partial M = S^1$ and
assume $s>2$. The exponential map of the $L^2$ metric \eqref{eq:L2} on $\mathscr{D}_\mu^s(M)$
is a nonlinear Fredholm map of index zero.
\end{theorem}
\hskip 0.5cm
In the next section we recall the basic setup from \cite{EMP06} and \cite{MP10}.
The proof of Theorem \ref{thm:Fred-b}
will be given in Sections \ref{sec:Fred} and \ref{sec:Fred-B}.
The key element of the proof involves deriving lower bounds for the invertible part of
the derivative $\text{d}{\exp}_e$ with respect to a suitably chosen Sobolev-type norm
defined on the space of stream functions on the manifold. The main idea here is that we have
an operator that is invertible because it is positive-definite in low Sobolev norms, but in the standard
higher Sobolev norms it is not due to boundary terms; however by weighting the coefficients differently
we can make the operator positive-definite in the new inner product up to leading order.

\section{The setup: Jacobi fields and the exponential map}
\label{sec:Back}

\hskip 0.5cm
We first collect a few well known facts about Fredholm mappings.
A bounded linear operator $L$ between Banach spaces is said to be
Fredholm if it has finite dimensional kernel and cokernel.
It then follows from the open mapping theorem that $\text{ran}\, L$ is closed.
$L$ is said to be semi-Fredholm if it has closed range and either its kernel or cokernel
is of finite dimension.
The index of $L$ is defined as
$
\mathrm{ind}\, L = \dim\ker L - \dim\mathrm{coker}\, L.
$
The set of semi-Fredholm operators is an open subset in the space of
all bounded linear operators and the index is a continuous function on this set
into $\mathbb{Z}\cup\left\{ \pm\infty\right\}$, cf. Kato \cite{Kato66}.
A $C^1$ map $f$ between Banach manifolds is called Fredholm
if its Fr\'echet derivative $df(p)$ is a Fredholm operator at each point $p$ in the domain of $f$.
If the domain is connected then the index of the derivative is by definition
the index of $f$, cf. Smale \cite{Smale65}.

\hskip 0.5cm
Let $\gamma$ be a geodesic in a Riemannian Hilbert manifold.
A point $q=\gamma(t)$ is said to be conjugate to $p=\gamma(0)$ if the derivative
${d}\exp_p(t\dot{\gamma}(0))$
is not an isomorphism considered as a linear operator between the tangent spaces at $p$ and $q$.
It is called monoconjugate if $d\exp_p(t\dot{\gamma}(0))$ fails to be injective
and epiconjugate if $d\exp_p(t\dot{\gamma}(0))$ fails to be surjective.
In general exponential maps of infinite dimensional Riemannian manifolds are not Fredholm.
For example, the antipodal points on the unit sphere in a Hilbert space with the induced metric
are conjugate along any great circle and the differential of the corresponding exponential map
has infinite dimensional kernel.
An ellipsoid constructed by Grossman \cite{Grossman65} provides
another example as it contains a sequence of monoconjugate points along a geodesic arc
converging to a limit point at which the derivative of the exponential map is injective but not surjective.
Such pathological phenomena are ruled out by the Fredholm property
because in this case monoconjugate and epiconjugate points must coincide, have finite multiplicities
and cannot cluster along finite geodesic segments.

\hskip 0.5cm
Let $M$ be a Riemannian manifold of dimension $n=2$ with boundary $\partial M$ and assume $s>2$.
%
%
%
%
%
%
Given any vector $v_0$ in $T_{e}\mathscr{D}_{\mu}^{s}$ let $\eta(t) = \exp_e(tv_0)$ be the geodesic of the $L^2$ metric
starting from the identity with velocity $v_0$.
The derivative of the exponential map at $tv_0$ can be expressed in terms of the Jacobi fields.
Since the curvature tensor $\mathcal{R}$ of the $L^2$ metric is bounded in the $H^s$ topology
it follows that the solutions of the Jacobi equation
\begin{equation} \label{eq:Jacobi}
J'' + \mathcal{R}(J,\dot{\eta})\dot{\eta} = 0
\end{equation}
along $\eta(t)$ with initial conditions
\begin{equation} \label{eq:Jacobi-initial}
J(0)=0,
\quad
J'(0) = w_0
\end{equation}
are unique and persist (as long as the geodesic is defined) by the standard ODE theory on Banach manifolds,
cf. \cite{Misiolek92}. Define the Jacobi field solution operator $\Phi_t$ by
\begin{equation} \label{eq:dexp}
w_0 \to \Phi_t w_0
=
d\exp_{e}(tv_0)tw_0 = J(t).
\end{equation}
%

\hskip 0.5cm
Next, recall that for any $\eta \in \mathscr{D}_\mu^s$ the group adjoint operator on
$T_{e}\mathscr{D}_{\mu}^{s}$ is given by $\mathrm{Ad}_{\eta} = dR_{\eta^{-1}}dL_{\eta}$
where
$R_{\eta}$ and $L_{\eta}$ denote the right and left translations by $\eta$.
Consequently, given any $v, w \in T_e\mathscr{D}_\mu^s$ we have
\begin{equation} \label{eq:Ad}
w \to \mathrm{Ad}_{\eta} w
=
\eta_\ast w
=
D\eta \circ \eta^{-1}( w \circ \eta^{-1})
\end{equation}
and the corresponding algebra adjoint operator
\begin{equation} \label{eq:ad}
\mathrm{ad}_{v}w = - [v,w].
\end{equation}
The associated coadjoint operators are defined using the $L^2$ inner product by
\begin{equation} \label{eq:grpcoAd-1}
\langle \mathrm{Ad}_{\eta}^{*}v,w\rangle _{L^{2}}
=
\langle v,\mathrm{Ad}_{\eta}w \rangle _{L^{2}}
\end{equation}
and
\begin{equation} \label{eq:grpcoad-1}
\langle \mathrm{ad}_{v}^{*}u,w \rangle _{L^{2}}
=
\langle u,\mathrm{ad}_{v}w \rangle _{L^2}
\end{equation}
for any $u, v$ and $w \in T_e\mathscr{D}_\mu^s$.
Our general strategy of the proof of Theorem 1 will be similar to that in the case when $M$ has no boundary. The proofs of the following result can be found in \cite{MP10}.
\begin{proposition} \label{prop:OLD}
Let $v_0 \in T_e\mathscr{D}_\mu^s$ and let $\eta(t)$ be the geodesic of the $L^2$ metric \eqref{eq:L2}
in $\mathscr{D}_{\mu}^{s}$ starting from the identity $e$ with velocity $v_0$.
Then $\Phi_t$ defined in \eqref{eq:dexp} is a family of bounded linear operators from
$T_{e}\mathscr{D}_{\mu}^{s}$ to $T_{\eta(t)}\mathscr{D}_{\mu}^{s}$.
Furthermore, if $v_0 \in T_e\mathscr{D}_\mu^{s+1}$ then $\Phi_t$ can be represented as
\begin{equation} \label{eq:solop}
\Phi_t = D\eta(t)\big( \Omega_t-\Gamma_t \big)
\end{equation}
where $\Omega_t$ and $\Gamma_t$ are bounded operators on $T_e\mathscr{D}_\mu^s$
given by
\begin{align}
&\Omega_t
=
\int_{0}^{t}\mathrm{Ad}_{\eta(\tau)^{-1}}\mathrm{Ad}_{\eta(\tau)^{-1}}^{*}\, d\tau
\label{eq:Omega} \\  \label{eq:Gamma}
&\Gamma_t
=
\int_{0}^{t}\mathrm{Ad}_{\eta(\tau)^{-1}}K_{v(\tau)}dR_{\eta^{-1}(\tau)}\Phi_\tau \, d\tau
\end{align}
and $K_v$ is a compact operator on $T_e\mathscr{D}_\mu^s$ given by
\begin{align} \label{eq:K}
w \to K_{v(t)} w = \mathrm{ad}_w^\ast v(t),
\quad
w \in T_e\mathscr{D}_\mu^s
\end{align}
and where $v(t)$ is the solution of the Cauchy problem \eqref{eq:euler}-\eqref{eq:eulerIC}.
\end{proposition}
\begin{proof}
See \cite{EMP06}, Prop. 4 and Prop. 8.
\end{proof}
\begin{rem} \upshape
Note that the decomposition \eqref{eq:solop}-\eqref{eq:K} must be applied with care.
This is due to the loss of derivatives involved in calculating the differential of the left translation operator
$\xi \to L_\eta \xi$ and consequently of the adjoint operator in \eqref{eq:Ad}. This is why we consider $v_0$, and hence $\eta(t)$, in $H^{s+1}$ rather than $H^s$.
\end{rem}

\hskip 0.5cm
As mentioned in the Introduction we also have the following
\begin{proposition} \label{prop:OLDFred}
For any $v_0 \in T_e\mathscr{D}_\mu^s$ the derivative $d\exp_e(tv_0)$ extends to
a Fredholm operator on the $L^2$-completions
$\overline{T_e\mathscr{D}_\mu}^{_{L^2}}$
and~
$\overline{T_{\exp_e(tv_0)}\mathscr{D}_\mu}^{_{L^2}}$.
\end{proposition}
\begin{proof}
A detailed proof may be found in \cite{EMP06}, Thm. 2, but the main idea is as follows. The operator \eqref{eq:Omega} is invertible on $\overline{T_e\mathscr{D}_\mu}^{_{L^2}}$. This follows from Lemma \ref{lem:L2} below and self-adjointness in the $L^{2}$ inner product. Compactness of the operator \eqref{eq:Gamma} on $\overline{T_e\mathscr{D}_\mu}^{_{L^2}}$ follows from compactness of the operator $K_{v}$, and hence compactness of the composition appearing under the integral in \eqref{eq:Gamma}, and finally from viewing the integral as a limit of sums of compact operators. This represents $d\exp_e(tv_0)$ as the sum of an invertible operator and compact operator which implies $d\exp_e(tv_0)$ is Fredholm of index zero.
\end{proof}
\hskip 0.5cm
In particular, it follows that monoconjugate points along $\eta(t)$ in $\mathscr{D}^s_\mu$
have finite multiplicity.

\section{Proof of Theorem 1: Preliminary Estimates}
\label{sec:Fred}

\hskip 0.5cm
To show that the $L^2$ exponential map on $\mathscr{D}^s_\mu$ is a Fredholm map we will prove that
for each $t>0$ its derivative $\Phi_t$ is a bounded Fredholm operator from $T_e\mathscr{D}^s_e$
to $T_{\eta(t)}\mathscr{D}^s_\mu$; that is, $\Phi_{t}$ can be expressed as the sum of an invertible operator and a compact operator on $T_e\mathcal{D}^s_{\mu}$.
We will assume that the initial divergence free vector field $v_0$ in \eqref{eq:dexp} is of class $C^\infty$.
The general $H^s$ case will then follow from a density argument, just as in \cite{EMP06}. Compactness of \eqref{eq:Gamma} then follows from Proposition \ref{prop:OLD} as described in the proof of Proposition \ref{prop:OLDFred}.

\hskip 0.5cm
It remains to prove that the operator $\Omega_t$, defined by \eqref{eq:Omega}, is invertible
on the tangent space $T_e\mathscr{D}_\mu^s$. We begin with an $L^2$ estimate which is straightforward.
\begin{lemma} \label{lem:L2}
Assume $s>2$. Given $v_0 \in T_e\mathscr{D}_\mu^s$ let $\eta(t)=\exp_e tv_0$
be the corresponding $L^2$ geodesic.
For any $w \in T_e\mathscr{D}_\mu^s$ and any $t \geq 0$ we have
\begin{align} \label{eq:OmegaL2}
\langle w, \Omega_t w \rangle_{L^2}
\ge
C_t \| w \|_{L^2}^2
\end{align}
where $C_t = \int_0^t \| D\eta(\tau) \|_\infty^{-2} d\tau$.
\end{lemma}
\begin{proof}
From \eqref{eq:Omega} we compute
\begin{align*}
\langle w, \Omega_t w \rangle_{L^2}
&=
\int_0^t
\big\langle w, \mathrm{Ad}_{\eta(\tau)^{-1}}\mathrm{Ad}_{\eta(\tau)^{-1}}^\ast w \big\rangle_{L^2} d\tau
\\
&=
\int_0^t \big\| \mathrm{Ad}_{\eta(\tau)^{-1}}^\ast w \big\|_{L^2}^2 d\tau
\geq
\| w \|_{L^2}^2 \int_0^t \big\| \mathrm{Ad}_{\eta(\tau)}^\ast \|_{L(L^2)}^{-2} d\tau
\end{align*}
and since $\mathrm{Ad}_{\eta}^\ast$ is an $L^2$ adjoint of $\mathrm{Ad}_{\eta}$,
formula \eqref{eq:Ad} implies
$$
\| \mathrm{Ad}_{\eta(t)}^\ast \|_{L(L^2)}^2
=
\| \mathrm{Ad}_{\eta(t)} \|_{L(L^2)}^2
\lesssim
\| D\eta(t)^{\transpose} D\eta(t) \|_\infty
$$
which gives \eqref{eq:OmegaL2}.
\end{proof}

\hskip 0.5cm
Next, we proceed to derive the estimate in $H^s$ norms. It will be convenient to work with
stream functions on $M = S^1 \times [0,\infty)$.
More precisely, we introduce the space
\begin{align*}
\mathcal{F}^{s+1}(M) =
\big\{ f \in H^{s+1}(M):
~&\text{$f$ vanishes rapidly as $y \to \infty$ and $f\vert_{\partial M}=0$} \big\}
\end{align*}
so that
$$
T_e\Diffmu^s(M)
=
\big\{
v_f = {-}\partial_y f \tfrac{\partial}{\partial x} {+} \partial_x f \tfrac{\partial}{\partial y}: f \in \mathcal{F}^{s+1}(M)
\big\}.
$$
From \eqref{eq:Ad} and \eqref{eq:grpcoAd-1} we have
\begin{align} \label{eq:Lambda_inv}
\mathrm{Ad}_{\eta^{-1}(t)} \mathrm{Ad}_{\eta^{-1}(t)}^\ast v_f
=
v_{\Lambda_t^{-1} f}
\quad \text{where} \quad
\Lambda_t = \Delta_0^{-1} \circ R_{\eta(t)} \circ \Delta \circ R_{\eta^{-1}(t)}.
\end{align}
Here $f=\Delta_0^{-1} g$ is the unique solution of the Dirichlet problem
$\Delta f = g$ with $f\vert_{\partial M}=0$ and $R_{\eta}g = g \circ \eta$.
Our goal therefore reduces to establishing the following
\begin{equation} \label{eq:OMEGA}
\text{Claim: \it For any $t>0$ the operator} \;\;
f \to \widehat\Omega_t f = \int_0^t \Lambda_\tau^{-1} f \, d\tau
\;\;
\text{\it is invertible on $\mathcal{F}^{s+1}(M)$}.
\end{equation}

\hskip 0.5cm
To this end we will proceed indirectly since the formula for $\Lambda_t$ is somewhat simpler to work with
than the formula for the inverse
$\Lambda^{-1}_t = R_{\eta(t)}\circ\Delta_0^{-1}\circ R_{\eta^{-1}(t)} \circ \Delta$.
Our approach to proving the claim \eqref{eq:OMEGA} is as follows. For some constants $B_{0},...,B_{s}$ we define a semi-inner product
on $\mathcal{F}^{s+1}$ by
\begin{equation}
\llangle f,g \rrangle_{s+1}
=
\sum_{j=0}^s B_j \langle \partial_x^j \partial_y^{s-j} \grad f, \partial_x^j \partial_y^{s-j} \grad g \rangle_{L^2}
\end{equation}
with associated semi-norm $\| f \|_{s+1} = \llangle f, f \rrangle_{s+1}^{1/2}$. Then, we show that the constants $B_{0},...,B_{s}$ can be chosen to be positive so that this semi-norm defines a norm, equivalent to the Sobolev $H^{s+1}$ norm, with

\begin{equation}
\llangle f,g \rrangle_{s+1}
\ge
K \lVert f \rVert^2_{s+1} - C \lVert f \rVert_{s+1} \lVert f \rVert_{\dot{H}^s}
\end{equation}

for $g=\Lambda_{t}f$, where $\lVert f \rVert_{\dot{H}^s}$ denotes the homogeneous Sobolev norm defined by \eqref{dot-H} below.  Applying this estimate to $f=\Lambda_t^{-1}g$ shows that $\widehat{\Omega}_{t}$ has closed range on $\mathcal{F}^{s+1}$. This, together with Lemma 1, implies that $\widehat{\Omega}_{t}$ is semi-Fredholm with trivial kernel whose index at $t=0$ is zero. Since the index is constant on connected component of the space of semi-Fredholm operators (cf. \cite{Kato66}) we conclude that the index is always zero so that $\widehat{\Omega}_{t}$ has trivial cokernel and is therefore invertible.
To carry out this plan we need to estimate the boundary terms and this is our main goal here; the analysis of these terms begins in Proposition 3 below.
The following observation will be convenient in the calculations.
\begin{lemma} \label{twistedlaplacianlemma}
Let $f\colon M\to \mathbb{R}$ be any $H^{s+1}$ function vanishing at infinity and let $\eta \in \Diffmu(M)$
be a smooth area-preserving diffeomorphism. Then
\begin{equation}\label{lambdaformula}
\Laplacian(f\circ\eta^{-1})\circ\eta = \diver{(G_\eta \nabla f)}
\qquad \text{where} \quad
G_\eta = (D\eta^{\transpose}D\eta)^{-1} = \left( \begin{matrix} \lvert\partial_y\eta\rvert^2 & -\langle \partial_x\eta, \partial_y\eta\rangle \\
-\langle \partial_x\eta, \partial_y\eta\rangle & \lvert \partial_x\eta\rvert^2\end{matrix}\right).
\end{equation}
\end{lemma}
\begin{proof}
Let $g\colon M\to \mathbb{R}$ be any smooth function that vanishes on $\partial M$.
Then by the change of variables formula and an integration by parts we have
$$
\int_M g\Laplacian (f\circ \eta^{-1} ) \circ \eta \, dxdy
=
\int_M g\circ\eta^{-1} \, \Laplacian ( f\circ\eta^{-1} ) \, dxdy
=
-\int_M \langle \grad( g\circ\eta^{-1} ), \grad( f\circ\eta^{-1} ) \rangle \, dxdy.
$$
Observe that $\grad ( f\circ\eta^{-1} ) = (D\eta^{-1})^{\transpose} \grad f \circ \eta^{-1}$,
so that
\begin{align*}
\int_M g \Laplacian(f\circ \eta^{-1}) \circ \eta \, dxdy
&=
- \int_M
\langle (D\eta^{-1})^{\transpose} \grad f \circ\eta^{-1}, (D\eta^{-1})^{\transpose} \grad g \circ\eta^{-1} \rangle
\, dxdy
\\
&=
-\int_M \langle (D\eta^{\transpose})^{-1} \grad f, (D\eta^{\transpose})^{-1} \grad g \rangle \, dxdy,
\end{align*}
using the change of variables formula once more.
We conclude that
\begin{align*}
\int_M g\Laplacian(f\circ \eta^{-1} )\circ \eta \, dxdy
&=
-\int_M \langle (D\eta^{\transpose}D\eta)^{-1} \grad f, \grad g \rangle \, dxdy
\\
&=
\int_M g\diver{ \big( (D\eta^{\transpose}D\eta)^{-1} \grad f} \big) \, dxdy,
\end{align*}
and since this is true for every smooth $g$ we deduce \eqref{lambdaformula}.
\end{proof}

\hskip 0.5cm

The following inequality appears in \cite{EMP06} but without the boundary terms.

\begin{proposition}\label{gradientinequalityprop}
Let $\eta\in \Diffmu(M)$ be a $C^\infty$-smooth area-preserving diffeomorphism and let
$g = \Lambda f$ where $f \in \mathcal{F}^{s+1}(M)$.
For any nonnegative integers $m$ and $n$ let $f_{m,n} = \partial_x^m \partial_y^n f$
and $g_{m,n} = \partial_x^m \partial_y^n g$.
Then we have\footnote{Here we agree to the convention that the boundary integral is zero if $n=0$.}
\begin{multline}\label{firstpositive}
\langle \grad f_{m,n}, \grad g_{m,n} \rangle_{L^2}
\ge
K_{\eta} \lVert \grad f_{m,n} \rVert^2_{L^2}
-
C \lVert \eta\rVert_{C^{m+n+1}}^2 \lVert \grad f_{m,n} \rVert_{L^2} \lVert f \rVert_{\dot{H}^{m+n}}
\\
- \int_{\partial M}
f_{m+1,n}
\partial_x^m \partial_y^{n-1}
\Big( \partial_x g - \lvert\partial_y\eta\rvert^2\partial_x f + \langle \partial_x\eta,\partial_y\eta \rangle \partial_yf \Big)
\, dx
\end{multline}
where $K_{\eta} = \| D\eta \|_\infty^{-2}$, with $C$ some constant independent of $\eta$, and
\begin{equation}\label{dot-H}
\|f\|_{\dot{H}^{s+1}} = \sum_{0 \leq i+j \leq s} \big\|\nabla\partial_x^i \partial_y^j f \big\|_{L^2}.
\end{equation}
\end{proposition}
\begin{proof}
We start with an integration by parts to obtain
\begin{equation}\label{firstintegration}
\begin{split}
\langle \grad f_{m,n}, \grad g_{m,n} \rangle_{L^2}
&=
\int_M \langle \grad f_{m,n}, \grad g_{m,n} \rangle\, dxdy
\\
&=
\int_M \diver{(f_{m,n} \grad g_{m,n})} \, dxdy
-
\int_M f_{m,n} \partial_x^m \partial_y^n \Laplacian g \, dxdy.
\end{split}
\end{equation}
Using Lemma \ref{twistedlaplacianlemma} we further rewrite the above as
\begin{equation}\label{secondintegration}
\begin{split}
&=
-\int_{\partial M} f_{m,n} g_{m,n+1} \, dx
-
\int_M f_{m,n}  \diver{(\partial_x^m \partial_y^n G_\eta\nabla f)} \, dxdy
\\
&=
\int_{\partial M} f_{m,n}\vert_{\partial M} \big(
-g_{m,n+1} - \partial_x^m \partial_y^n \langle G_\eta\nabla f, \nu \rangle\big)\vert_{\partial M} \, dx
+
\int_M \langle \grad f_{m,n}, \partial_x^m \partial_y^n G_\eta\nabla f \rangle \, dxdy,
\end{split}
\end{equation}
where we again used the divergence theorem and where the outward unit normal is $\nu=(0,-1)$.

\hskip 0.5cm
We proceed to analyze these terms separately. Observe that $G_\eta = (D\eta^{\transpose}D\eta)^{-1}$
is a positive-definite matrix and the last term can be written as
\begin{align*}
\langle \grad f_{m,n}, \partial_x^m \partial_y^n G_\eta\nabla f \rangle
&=
\int_M \langle \grad f_{m,n}, G_\eta\partial_x^m\partial_y^n \nabla f \rangle \, dxdy
+
\int_M \langle \grad f_{m,n}, [\partial_x^m\partial_y^n, G_\eta] \grad f \, dxdy
\\
&\ge
\int_M \lvert
(D\eta^{\transpose})^{-1} \grad f_{m,n}\rvert^2 \, dxdy
-
\lVert \grad f_{m,n}\rVert_{L^2} \big\lVert [\partial_x^m\partial_y^n, G_\eta] \grad f \big\rVert_{L^2}.
\end{align*}
Since $G_\eta$ is a matrix of smooth functions, the commutator with any differential operator of order $m+n$
is a differential operator of lower order with coefficients involving derivatives of $\eta$ up to order $m+n+1$
at most. Hence we have an estimate
$$
\big\lVert [\partial_x^m\partial_y^n, G_\eta]\grad f \big\rVert_{L^2}
\le
C\lVert \eta\rVert^2_{C^{m+n+1}} \lVert f\rVert_{\dot{H}^{m+n}(M)}
$$
with $\| \cdot\|_{\dot{H}^{m+n}}$ denoting the Sobolev $H^{m+n}$ norm that omits the lowest-order terms
(in other words, the $H^{m+n-1}$ norm of the gradient).
On the other hand we have
$$
\int_M \lvert (D\eta^{\transpose})^{-1} \grad f_{m,n} \rvert^2 \, dxdy
\ge
K_{\eta} \lVert \grad f_{m,n} \rVert^2_{L^2}
$$
where $K_{\eta}$ is the infimum over $M$ of the eigenvalues of
$G_\eta = (D\eta^{\transpose}D\eta)^{-1}$
and minimizing over all of $M$ gives the formula for $K_{\eta}$.

\hskip 0.5cm
Next, consider the boundary term in \eqref{secondintegration} given by
$$
\int_{\partial M} f_{m,n} \,
\partial_x^m\partial_y^n \big(
-\partial_y g + | \partial_x\eta |^2 \partial_y f - \langle \partial_x\eta,\partial_y\eta \rangle \partial_x f
\big) \, dx.
$$
Since $f\vert_{\partial M}=0$, we know that $f_{m,0}\vert_{\partial M}=0$ so that this term vanishes if $n=0$.
If $n\ge 1$ then we can use the equation $\Laplacian g = \diver{(G_{\eta} \nabla f)}$ to simplify
\begin{equation} \label{PDEboundary}
\partial_y (-\partial_y g + | \partial_x\eta |^2 \partial_y f -\langle\partial_x\eta,\partial_y\eta \rangle\partial_x f)
=
\partial_x (\partial_x g - | \partial_y\eta |^2 \partial_x f +\langle \partial_x\eta,\partial_y\eta \rangle\partial_y f)
\end{equation}
so that the boundary term becomes the last term of \eqref{firstpositive} after an integration by parts in $x$.
\end{proof}

Before we proceed further with the proof in full generality, let us illustrate the basic idea
with a simple explicit example.

\begin{ex} \upshape
Consider the diffeomorphism
$$ \eta(x,y) = (x+\phi(y), y),$$
for some function $\phi$ which decays as $y\to\infty$. This is a shear flow, and the function
$ \eta(t,x,y) = (x+t\phi(y),y)$ is a solution of the inviscid Euler equation with steady velocity field
$ u(x,y) = \phi(y) e_x$. The matrix $G_{\eta}$ is given by
$$ G_\eta = (D\eta^{\transpose}D\eta)^{-1} = \left( \begin{matrix} 1+\phi'(y)^2 & -\phi'(y) \\
-\phi'(y) & 1\end{matrix}\right).$$

Consider the $H^2$ norm on vector fields, corresponding to the $\dot{H}^3$ norm on stream functions.
(This is the first interesting case, as $H^1$ on vector fields has an accidental cancellation\footnote{See Lemma \ref{gboundlemma}.} and $L^2$
on vector fields is the weak case already discussed.) We now consider the three terms that together make
up the $\dot{H}^3$ inner product, and weight them with positive constants $B_i$: we get
\begin{equation}\label{H3norm}
\llangle f,g\rrangle_3 =
 B_0 \langle \nabla f_{yy}, \nabla g_{yy}\rangle_{L^2} + B_1 \langle \nabla f_{xy}, \nabla g_{xy}\rangle_{L^2} + B_2\langle \nabla f_{xx}, \nabla g_{xx}\rangle_{L^2}.
\end{equation}

By Proposition \ref{gradientinequalityprop}, the leading-order terms in \eqref{H3norm} are
\begin{align*}
\llangle f,g\rrangle_3 &= B_0 \lVert G_{\eta} \grad f_{yy}\rVert^2_{L^2} + B_1 \lVert G_{\eta}\grad f_{xy}\rVert^2_{L^2} + B_2 \lVert G_{\eta}\grad f_{xx}\rVert^2_{L^2} \\
&\qquad\qquad - \int_{S^1} B_0 f_{xyy}(x,0) \big(g_{xy}(x,0) - (1+\phi'(0)^2) f_{xy}(x,0) + \phi'(0) f_{yy}(x,0)) \\
&\qquad\qquad - \int_{S^1} B_1 f_{xxy}(x,0) \big(g_{xx}(x,0) - (1+\phi'(0)^2) f_{xx}(x,0) + \phi'(0) f_{xy}(x,0)) \\
&= B_0 \lVert G_{\eta} \grad f_{yy}\rVert^2_{L^2} + B_1 \lVert G_{\eta}\grad f_{xy}\rVert^2_{L^2} + B_2 \lVert G_{\eta}\grad f_{xx}\rVert^2_{L^2} \\
&\qquad\qquad - B_0\int_{S^1} f_{xyy}(x,0) g_{xy}(x,0) \, dx + B_0 (1+\phi'(0)^2) \int_{S^1} f_{xyy}(x,0) f_{xy}(x,0) \, dx,
\end{align*}
after integrating by parts.

Now we can estimate the second boundary term as follows:
\begin{align*}
\int_{S^1} f_{xyy}(x,0) f_{xy}(x,0) \, dx &= -\int_{M} \partial_y(f_{xyy} f_{xy}) \, dA
= -\int_M (f_{xyyy} f_{xy} + f_{xyy}^2) \, dA \\
&= \int_M (f_{yyy} f_{xxy} - f_{xyy}^2) \, dA \\
&\ge -\frac{\varepsilon}{2} \int_M f_{yyy}^2 \, dA - \frac{1}{2\varepsilon} \int_M f_{xxy}^2 \, dA - \int_M f_{xyy}^2 \, dA.
\end{align*}
Likewise we can estimate the first boundary term and combine with an upper bound on $g$ given by $\lVert \grad g_{xy} \rVert\le C_{\eta} \lVert \grad f_{xy}\rVert$ to get similar terms, ignoring lower-order terms.
We end up with
\begin{align*}
\llangle f,g \rrangle_3 &\ge B_0\big(K_{\eta} - \tfrac{\varepsilon}{2}\tilde{C}_{\eta}) \big)\lVert f_{yyy}\rVert^2
+ \big( K_{\eta} (B_0+B_1) - \tilde{C}_{\eta} B_0 \big) \lVert f_{xyy}\rVert^2 \\
&\qquad\qquad + \big( K_{\eta} (B_1+B_2) - \tfrac{1}{2\varepsilon} \tilde{C}_{\eta} B_0\big) \lVert f_{xxy}\rVert^2
+ K_{\eta} B_2 \lVert f_{xxx}\rVert^2,
\end{align*}
where $\tilde{C}_{\eta} = 1+\phi'(0)^2 + C_{\eta}$, ignoring all lower-order terms.

From this formula we see that if $\varepsilon$ is sufficiently small and $B_2\gg B_1\gg B_0>0$, then we can arrange
$$ \llangle f, g\rrangle_3 \ge c_{\eta} \lVert f\rVert^2_3$$
for some very small $c_{\eta}>0$, up to lower-order terms. This implies that the map $f\mapsto g=\Lambda f$ has closed range, which leads to invertibility. This is the main idea of the proof we give in the next section.
\end{ex}
%

\section{Proof of Theorem 1: Estimates at the Boundary}
\label{sec:Fred-B}

\hskip 0.5cm
We now need to estimate the boundary terms appearing in equation \eqref{firstpositive}.
The following lemma simplifies many of the calculations.
\begin{lemma} \label{xderivativeboundarylemma}
For any $H^1$ functions $f$ and $g$ on $M$ vanishing at infinity, we have
$$
\int_{\partial M} f \partial_x g \, dx \le \lVert \grad f\rVert_{L^2} \lVert \grad g\rVert_{L^2}.
$$
\end{lemma}

\begin{proof}
A straightforward computation gives
\begin{align*}
\int_{\partial M} f(x,0) \partial_x g(x,0) \, dx
&=
-\int_{\partial M} \int_0^\infty \frac{\partial}{\partial y}\Big( f(x,y) \partial_x g(x,y) \Big) \, dydx
\\
&=
-\int_M \partial_y f \partial_x g \, dxdy + \int_M \partial_x f \partial_y g \, dxdy
\\
&= \Big\langle (-\partial_y f, \partial_x f), (\partial_xg, \partial_yg)\Big\rangle_{L^2} \le
\lVert \grad f \rVert_{L^2} \lVert \grad g \rVert_{L^2}.
\end{align*}
\end{proof}

Now we estimate the boundary terms in Proposition \ref{gradientinequalityprop} in terms of
norms on the entire space $M$.

\begin{proposition} \label{boundaryestimate}
Let $\eta \in \mathscr{D}_\mu(M)$.
If $f \in \mathcal{F}^{s+1}(M)$ and $g = \Lambda f$ then given any $m\ge 0$ and $n\ge 1$
the boundary terms in \eqref{firstpositive} can be estimated by
\begin{align}
\int_{\partial M} f_{m+1,n} g_{m+1,n-1} \, dx
&\le
\lVert \grad f_{m,n} \rVert_{L^2} \lVert \grad g_{m+1,n-1} \rVert_{L^2}
\qquad (\text{$n>1$}) \label{fgboundary} \\
\int_{\partial M} f_{m+1,n} \partial_x^m\partial_y^{n-1} \big( \lvert \partial_y\eta \rvert^2 \partial_x f \big) \,dx
&\le
\lVert \eta \rVert^2_{C^1} \lVert \grad f_{m,n} \rVert_{L^2} \lVert \grad f_{m+1,n-1} \rVert_{L^2}
\nonumber \\
&\qquad\qquad\quad
+
C \lVert \eta \rVert_{C^{m+n+1}}^2 \lVert \grad f_{m,n} \rVert_{L^2} \lVert f \rVert_{\dot{H}^{m+n}}
\label{ffxboundary} \\
\int_{\partial M}
f_{m+1,n} \partial_x^m \partial_y^{n-1}\big( \langle \partial_x\eta,\partial_y\eta \rangle \partial_y f \big)
\, dx
&\le
C \lVert \eta \rVert^2_{C^{m+n+1}} \lVert \grad f_{m,n} \rVert_{L^2} \lVert f \rVert_{\dot{H}^{m+n}}
\label{ffyboundary}
\end{align}
where $C>0$ is independent of $\eta$ and the $\dot{H}^{m+n}$ norm is defined in \eqref{dot-H}.
\end{proposition}
\begin{proof}
Inequality \eqref{fgboundary} follows at once from Lemma \ref{xderivativeboundarylemma}.
To estimate \eqref{ffxboundary} we use Lemma \ref{xderivativeboundarylemma} and the Leibniz rule to get
\begin{align*}
&\int_{\partial M} f_{m+1,n} \partial_x^m\partial_y^{n-1} \big( \lvert \partial_y \eta \rvert^2 \partial_x f \big) \,dx
\le \lVert \grad f_{m,n}\rVert_{L^2} \lVert \grad \partial_x^m\partial_y^{n-1}\big( \lvert \partial_y\eta\rvert^2 \partial_x f\big)\rVert_{L^2}\\
&\qquad\qquad \le
\lVert \eta \rVert^2_{C^1} \lVert \grad f_{m+1,n-1} \rVert_{L^2} \|\nabla f_{m,n}\|_{L^2}
+
C \lVert \eta \rVert^2_{C^{m+n+1}} \|f\|_{\dot{H}^{m+n}} \lVert \grad f_{m,n} \rVert_{L^2}
\end{align*}
%

\hskip 0.5cm
For \eqref{ffyboundary} we use a trick on the highest-order term to improve over
Lemma \ref{xderivativeboundarylemma}:
\begin{align*}
\int_{\partial M}
f_{m+1,n} \partial_x^m \partial_y^{n-1} \big( \langle \partial_x\eta,\partial_y\eta \rangle \partial_y f \big)
\, dx
&\le
\int_{\partial M}
f_{m+1,n} \langle \partial_x\eta,\partial_y\eta \rangle f_{m,n} \, dx
\\
&\qquad\qquad\qquad +
C \|\eta\|_{C^{m+n+1}}^2 \|f\|_{\dot{H}^{m+n}} \|\nabla f_{m,n}\|_{L^2}
\end{align*}
%
Now integrating the first term on the right by parts, we obtain (using Lemma \ref{xderivativeboundarylemma})
\begin{align*}
\int_{\partial M} f_{m+1,n} \langle \partial_x\eta, \partial_y\eta\rangle f_{m,n} \, dx
&= \frac{1}{2} \int_{\partial M} \langle \partial_x\eta, \partial_y\eta\rangle \, \partial_x(f_{m,n}^2) \, dx = -\frac{1}{2} \int_{\partial M} \partial_x\langle \partial_x\eta, \partial_y\eta\rangle f_{m,n}^2\,dx \\
&\lesssim \lVert \eta\rVert_{C^2} \lVert f_{m,n}\rVert_{L^2} \lVert \grad f_{m,n}\rVert_{L^2} \lesssim \lVert \eta\rVert_{C^2} \lVert f\rVert_{\dot{H}^{m+n}} \lVert \grad f_{m,n}\rVert_{L^2},
\end{align*}
and thus this term folds into our previous term.
%
\end{proof}

\hskip 0.5cm
It remains to estimate the term $\lVert \grad g_{m+1,n-1}\rVert_{L^2}$ in terms of
a suitable norm of $f$.
\begin{lemma} \label{gboundlemma}
Let $\eta$, $f$ and $g = \Lambda f$ be as in Proposition \ref{boundaryestimate}.
For any integers $m \ge 1$ and $n > 1$ we have
\begin{multline} \label{gbound}
\lVert \grad g_{m,n} \rVert_{L^2}
\le
\lVert \grad g_{m+1,n-1} \rVert_{L^2} + \lVert \eta \rVert^2_{C^1} \lVert \grad f_{m+1,n-1} \rVert_{L^2}
+
\lVert \eta \rVert^2_{C^1} \lVert \grad f_{m,n} \rVert_{L^2} \\
+ C \lVert \eta\rVert^2_{C^{m+n+1}} \lVert f \rVert_{\dot{H}^{m+n}}
\end{multline}
while for $n=0$ or $n=1$ we have
\begin{equation} \label{gboundxonly}
\lVert \grad g_{m,n} \rVert_{L^2}
\le
\lVert \eta \rVert^2_{C^1} \lVert \grad f_{m,n} \rVert_{L^2} + C \lVert \eta\rVert^2_{C^{m+n+1}} \lVert f \rVert_{\dot{H}^{m+n}}
\end{equation}
where $C$ is a constant depending on $m$ and $n$ but not on $\eta$.
\end{lemma}
\begin{proof}
Integrating by parts as in \eqref{firstintegration}, we have
\begin{equation} \label{gintegration}
\begin{split}
\lVert \grad g_{m,n}\rVert^2_{L^2}
&=
\int_M \diver{ (g_{m,n} \grad g_{m,n} )} \, dxdy - \int_M g_{m,n} \Delta g_{m,n} \, dxdy
\\
&=
\int_{\partial M}g_{m,n} \partial_x^m \partial_y^n (-\partial_y g - \langle G_\eta\nabla f, \nu \rangle ) \, dx
+
\langle \grad g_{m,n}, \partial_x^m\partial_y^n G_\eta\nabla f \rangle_{L^2}.
\end{split}
\end{equation}

\hskip 0.5cm
We first consider the case when $m \geq 1$ and $n>1$.
Since
$\langle G_\eta\nabla f, \nu \rangle
=
\langle\partial_x\eta,\partial_y\eta\rangle \partial_x f
-
|\partial_x\eta|^2 \partial_y f$
from \eqref{PDEboundary} we get
$$
\partial_y (-\partial_y g - \langle G_\eta \nabla f, \nu \rangle )
=
\partial_x (\partial_x g - |\partial_y\eta|^2 \partial_x f + \langle\partial_x\eta,\partial_y\eta\rangle\partial_y f )
$$
Using this identity and integrating by parts in $x$ the right hand side of \eqref{gintegration} becomes
\begin{align} \nonumber
&=
-\int_{\partial M} g_{m+1,n} \partial_x^m\partial_y^{n-1} (
\partial_x g - |\partial_y\eta|^2 \partial_x f + \langle\partial_x\eta,\partial_y\eta\rangle\partial_y f
) \, dx
+
\langle\nabla g_{m,n}, \partial_x^m\partial_y^n G_\eta \nabla f \rangle_{L^2}
\\ \nonumber
&=
-\int_{\partial M} g_{m+1,n} g_{m+1,n-1} dx
+
\int_{\partial M} |\partial_y\eta|^2 g_{m+1,n} f_{m+1,n-1} dx
-
\int_{\partial M} \langle\partial_x\eta,\partial_y\eta\rangle g_{m+1,n} f_{m,n} dx
\\ \nonumber
&\qquad\qquad\qquad\qquad\qquad\qquad\; +
\sum_{0<k+l<m+n} \int_{\partial M} \alpha_{kl} g_{m+1,n} f_{k,l} dx
+
\langle \nabla g_{m,n}, \partial_x^m\partial_y^n G_\eta\nabla f \rangle_{L^2}
\\ \nonumber
&= I + II + III + IV + V
\end{align}
where $\alpha_{ij}$ are functions depending on the derivatives up to order $m+n-1$
of $|\partial_y\eta|^2$ and $\langle\partial_x\eta,\partial_y\eta\rangle$ and the binomial coefficients.
Using Lemma \ref{xderivativeboundarylemma} and \eqref{dot-H} we have
\begin{align} \label{eq:123}
|I +II+III|
&\le
\|\nabla g_{m+1,n-1}\|_{L^2} \|\nabla g_{m,n}\|_{L^2}
+
\| \nabla g_{m,n}\|_{L^2} \| \nabla ( |\partial_y\eta|^2 f_{m+1,n-1}) \|_{L^2}
\\ \nonumber
&\hskip 5.9cm +
\|\nabla g_{m,n}\|_{L^2} \|\nabla (\langle\partial_x\eta,\partial_y\eta\rangle f_{m,n} ) \|_{L^2}
\\ \nonumber
&\le
\big(
\|\nabla g_{m+1,n-1}\|_{L^2}
+
\|\eta\|_{C^1}^2 \|\nabla f_{m+1,n-1}\|_{L^2} + \|\eta\|_{C^1}^2 \|\nabla f_{m,n}\|_{L^2}
\\ \nonumber
&\hskip 5.9cm
+
\|\eta\|_{C^2}^2 \|f\|_{\dot{H}^{m+n}}
\big) \|\nabla g_{m,n}\|_{L^2}
\end{align}
and similarly
\begin{equation} \label{eq:4}
|IV|
\lesssim
\sum_{0<k+l<m+n} \|\nabla g_{m,n} \|_{L^2} \| \nabla (\alpha_{kl} f_{k,l} )\|_{L^2}
\lesssim
\|\eta\|_{C^{m+n+1}}^2 \|\nabla g_{m,n}\|_{L^2} \| f \|_{\dot{H}^{m+n}}.
\end{equation}
Using the Cauchy-Schwarz inequality and the Leibniz rule
\begin{align} \label{eq:5}
|V|
&\le
\| \nabla g_{m,n} \|_{L^2} \| \partial_x^m\partial_y^n G_\eta \nabla f \|_{L^2}
\\ \nonumber
&\le
\| \eta \|_{C^1}^2 \| \nabla g_{m,n} \|_{L^2}  \| \nabla f_{m,n} \|_{L^2} +
C \| \nabla g_{m,n} \|_{L^2} \| \eta\|_{C^{m+n+1}}^2 \| f \|_{\dot{H}^{m+n}}.
\end{align}

\hskip 0.5cm
Combining \eqref{eq:123}, \eqref{eq:4} and \eqref{eq:5} we obtain \eqref{gbound},
as desired.

\hskip 0.5cm
Next, if $m \geq 1$ and $n=0$ then the boundary term in \eqref{gintegration} vanishes
since $g\vert_{\partial M} = 0$ and we have
\begin{align} \label{eq:n=0}
\| \nabla g_{m,0} \|_{L^2}^2
&\le
\| \nabla g_{m,0} \|_{L^2} \| \partial_x^m G_\eta \nabla f\|_{L^2}
\\ \nonumber
&\le
\lVert \eta\rVert_{C^1}^2 \lVert \grad g_{m,0}\rVert_{L^2} \lVert \grad f_{m,0}\rVert_{L^2} +
C\| \nabla g_{m,0}\|_{L^2} \| \eta \|_{C^{m+1}}^2 \| f \|_{\dot{H}^m}.
\end{align}

\hskip 0.5cm
Finally, if $m \geq 1$ and $n=1$ we use a trick to do a little better than \eqref{gbound}. Integrating by parts in \eqref{gintegration}
as before we have
\begin{align} \nonumber
\| \nabla g_{m,1} \|_{L^2}^2
&=
-\int_{\partial M} g_{m+1,1} \partial_x^{m+1} g \, dx
+
\int_{\partial M} g_{m+1,1} \partial_x^m (|\partial_y \eta|^2 \partial_x f ) \, dx
\\ \nonumber
&\qquad\qquad\quad -
\int_{\partial M} g_{m+1,1} \partial_x^m ( \langle\partial_x\eta, \partial_y\eta\rangle \partial_y f ) \, dx
+
\langle \nabla g_{m,1}, \partial_x^m \partial_y (G_\eta \nabla f) \rangle_{L^2}
\end{align}
The first two terms on the right hand side drop out since $g\vert_{\partial M} = f\vert_{\partial M} = 0$.
The remaining terms can be estimated using Lemma \ref{xderivativeboundarylemma} and
the Cauchy-Schwarz inequality as before
to get
\begin{equation} \label{eq:n=1}
\| \nabla g_{m,1} \|_{L^2}^2 \le \|\eta\|_{C^1}^2 \|\nabla g_{m,1}\|_{L^2}  \|\nabla f_{m,1}\|_{L^2}
+
C\| \nabla g_{m,1}\|_{L^2} \|\eta\|_{C^{m+2}} \| f\|_{\dot{H}^{m+1}}
\end{equation}
where we used the homogeneous norm \eqref{xderivativeboundarylemma}.
\end{proof}

\hskip 0.5cm
Our next task is to eliminate all $g$-terms on the right side of the basic inequality \eqref{gbound}.
\begin{proposition} \label{ggoneprop}
Let $\eta$, $f$ and $g=\Lambda f$ be as above.
For any $m \geq 1$ and $n \geq 0$ we have
\begin{equation} \label{ggone}
\lVert \grad g_{m,n} \rVert_{L^2}
\le \lVert \eta\rVert_{C^1}^2 \lVert \grad f_{m,n}\rVert_{L^2}
+ 2 \| \eta \|_{C^1}^2 \sum_{k=1}^{n-1} \lVert \grad f_{m+n-k,k} \rVert_{L^2}
+
C \| \eta \|_{C^{m+n+1}}^2 \lVert f \rVert_{\dot{H}^{m+n}}
\end{equation}
for some constant $C$ independent of $\eta$.
\end{proposition}
\begin{proof}
Adding and subtracting terms and using inequalities \eqref{gbound} and \eqref{gboundxonly}
we have
\begin{align*}
\lVert \grad g_{m,n} \rVert_{L^2}
&=
\lVert \nabla g_{m+n-1,1} \rVert_{L^2}
+
\sum_{k=1}^{n-1} \Big( \lVert \grad g_{m+n-k-1, k+1} \rVert_{L^2} - \lVert \grad g_{m+n-k,k} \rVert_{L^2} \Big)
\\
&\le \| \eta\|_{C^1}^2 \lVert \grad f_{m+n-1,1}\rVert_{L^2} + \lVert \eta\rVert_{C^1}^2 \sum_{k=1}^{n-1} \Big( \lVert \grad f_{m+n-k-1,k+1}\rVert_{L^2} + \lVert \grad f_{m+n-k,k}\rVert_{L^2} \Big) \\
&\qquad\qquad + C \lVert \eta\rVert^2_{C^{m+n+1}} \lVert f\rVert_{\dot{H}^{m+n}}, 
\end{align*}
and \eqref{ggone} follows.
%
\end{proof}

\hskip 0.5cm
Given an integer $s \geq 0$ and any numbers $B_0, \dots B_s$ define a semi-inner product
on the space of stream functions on $M$ by
\begin{equation} \label{Sobolevdoubledot}
\llangle f,g \rrangle_{s+1}
=
\sum_{j=0}^s B_j \langle \partial_x^j \partial_y^{s-j} \grad f, \partial_x^j \partial_y^{s-j} \grad g \rangle_{L^2}
\end{equation}
and the associated seminorm by $\| f \|_{s+1} = \llangle f, f \rrangle_{s+1}^{1/2}$.
\begin{proposition} \label{positivityprop}
Let $\eta(t)$ be a smooth curve in $\Diffmu(M)$.
Let $f$ be a smooth function with $f\vert_{\partial M}=0$ and let $g(t) = \Lambda_t f$.
Given $s \ge 1$ there exist positive coefficients $B_0, \dots B_s$ depending on $\eta$ but independent of $t$
such that for sufficiently small $\epsilon > 0$ we have
\begin{equation} \label{hereitfinallyis}
\llangle f,g \rrangle_{s+1}
\ge
K \lVert f \rVert^2_{s+1} - C \lVert f \rVert_{s+1} \lVert f \rVert_{\dot{H}^s}
\end{equation}
where $K>0$ and $C>0$ are constants depending on $\epsilon$ and $s$; in addition $K$ depends on the $L^{\infty}C^1_x$-norm and $C$ depends on the $L^\infty_t C^{s+1}_x$-norm\footnote{That is,
$\|\varphi\|_{L^\infty_t C^k_x} = \sup\limits_{0\leq \tau \leq t} \| \varphi(\tau) \|_{C^k}$.}
of $\eta$.
\end{proposition}
\begin{proof}
From Proposition \ref{gradientinequalityprop} for any $t\geq 0$ and any integers $m \geq 0$ and $n\geq 0$
with $m+n=s$ we have
\begin{align*}
\llangle \grad f_{mn}, \grad g_{mn}(t)\rrangle_{L^2}
&\ge
K_{\eta}  \lVert \grad f_{m,n} \rVert^2_{L^2}
-
C \lVert \eta \rVert^2_{L^\infty_t C^{m+n+1}} \lVert \grad f_{m,n} \rVert_{L^2} \lVert f \rVert_{\dot{H}^{m+n}}
\\
&\qquad\qquad -
\int_{\partial M} f_{m+1,n} \partial_x^m\partial_y^{n-1} \big(
\partial_x g - |\partial_y\eta|^2 \partial_x f + \langle \partial_x\eta,\partial_y\eta\rangle \partial_y f
\big) dx
\end{align*}
for some constant $C$ independent of $\eta$.
Note that by convention\footnote{See the footnote to Proposition \ref{gradientinequalityprop}.}
the integral over the boundary vanishes if $n=0$ and, furthermore, the first term of the integral
(corresponding to the factor $\partial_x g$) also vanishes if $n=1$ (since $g|_{\partial M} =0$ by assumption).
Therefore, using Proposition \ref{boundaryestimate},
we can now estimate the above expression from below by (using various constants, all of which we denote by $C$)
\begin{align*}
\langle \nabla f_{m,n}, \nabla g_{m,n} \rangle_{L^2} &\ge
K_\eta \| \nabla f_{m,n} \|_{L^2}^2
-
\| \nabla f_{m,n} \|_{L^2} \Big( \| \nabla g_{m+1,n-1} \|_{L^2} + \| \eta \|_{L^{\infty}_t C^1_x}^2 \| \nabla f_{m+1,n-1} \|_{L^2}\Big)
\\
&\qquad\qquad
-
C \| \eta \|_{L^{\infty}_t C^{m+n+1}_x}^2 \| f \|_{\dot{H}^{m+n}} \| \nabla f_{m,n} \|_{L^2}
\\
\end{align*}
and, with the help of Proposition \ref{ggoneprop} and rearranging and combining like terms,
estimate it even further by
\begin{align*}
\langle \nabla f_{m,n}, \nabla g_{m,n} \rangle_{L^2}&\ge K_{\eta} \lVert \grad f_{m,n}\rVert_{L^2}^2 - 2 \lVert \eta\rVert_{L^{\infty}_t C^1_x}^2 \lVert \grad f_{m,n}\rVert_{L^2}
\sum_{k=1}^{n-1} \lVert \grad f_{m+n-k,k}\rVert_{L^2} \\
&\qquad\qquad - C\lVert \eta\rVert_{L^{\infty}_t C^{m+n+1}_x}^2 \lVert \grad f_{m,n}\rVert_{L^2} \lVert f\rVert_{\dot{H}^{m+n}} \\
&\ge \Big(K_{\eta} - (s-1) \epsilon \lVert \eta\rVert_{L^{\infty}_t C^1_x}^2 \Big) \lVert \grad f_{m,n}\rVert_{L^2}^2
- \frac{1}{\epsilon} \lVert \eta\rVert_{L^{\infty}_t C^1_x}^2 \sum_{k=1}^{n-1} \lVert \grad f_{m+n-k,k}\rVert^2_{L^2} \\
&\qquad\qquad - C \lVert  \eta\rVert_{L^{\infty}_t C^{m+n+1}_x}^2 \lVert \grad f_{m,n}\rVert_{L^2} \lVert f\rVert_{\dot{H}^{m+n}}
\end{align*}
for any positive $\epsilon$.

%
%
%

Setting
\begin{align} \label{eq:CCC}
C_{\eta} &= C \| \eta \|_{L^\infty_t C^{s+1}_x}^2
\\
\label{eq:QQQ}
Q_\epsilon &= \frac{1}{\epsilon} \| \eta \|_{L^\infty_t C^1_x}^2
\\  \label{eq:KKK}
K_\epsilon
&=
\inf_{0\leq \tau \leq t} K_\eta - \frac{(s-1)}{2}\epsilon \| \eta \|_{L^\infty_t C^{s+1}_x}^2
\end{align}
and choosing
\begin{equation} \label{eq:eee}
0 < \epsilon
<
\dfrac{\displaystyle \inf_{0\leq \tau \leq t}K_\eta}{\displaystyle (s-1)\|\eta\|_{L^\infty_tC^1_x}^2}
\end{equation}
we therefore obtain
\begin{align} \label{eq:eef}
\langle \nabla f_{m,n}, \nabla g_{m,n} \rangle_{L^2}
\ge
K_\epsilon \| \nabla f_{m,n} \|_{L^2}^2
-
Q_\epsilon \sum_{k=1}^{n-1} \| \nabla f_{m+n-k,k} \|_{L^2}^2
-
C_{\eta} \| \nabla f_{m,n} \|_{L^2} \| f \|_{\dot{H}^{m+n}}
\end{align}
for any integers $m \geq 0$ and $n \geq 0$.

\hskip 0.5cm
Next, let $s \geq 1$. Given positive numbers $B_0, \dots, B_s$ (to be determined below)
consider
\begin{align*}
\llangle f,g \rrangle_{s+1}
&=
\sum_{j=0}^s B_j \langle \partial_x^j \partial_y^{s-j} \grad f, \partial_x^j \partial_y^{s-j} \grad g \rangle_{L^2}
= \sum_{j=0}^s B_j \langle \grad f_{j,s-j}, \grad g_{j,s-j}\rangle_{L^2}
\\
&\ge
\sum_{j=0}^s B_j \left(
K_\epsilon \| \nabla f_{j,s-j} \|_{L^2}^2
-
Q_\epsilon \sum_{k=1}^{s-j-1} \| \nabla f_{s-k,k} \|_{L^2}^2
-
C_\eta \| \nabla f_{j,s-j} \|_{L^2} \| f \|_{\dot{H}^s}
\right)
\\
&=
B_0 K_\epsilon \| \nabla f_{0,s} \|_{L^2}^2
+
B_s K_\epsilon \| \nabla f_{s,0} \|_{L^2}^2
+
\sum_{k=1}^{s-1} \left( K_\epsilon B_k - Q_\epsilon \sum_{j=0}^{k-1} B_j \right) \| \nabla f_{k,s-k} \|_{L^2}^2
\\
&\hskip 6.1cm -
C_{\eta,B}^s \| f \|_{s+1} \| f \|_{\dot{H}^s}
\end{align*}
where $\displaystyle C_{\eta,B}^s = C_{\eta} (s+1)^{1/2} \max_{0\leq j \leq s} \sqrt{B_j}$.
Now, for any $1 \leq k \leq s-1$ pick
\begin{equation} \label{eq:BBB}
B_k = \frac{2}{K_\epsilon} Q_\epsilon \sum_{j=0}^{k-1} B_j
\end{equation}
and set $B_0=B_s =1$.
Note that the solution of the recurrence equation in \eqref{eq:BBB} is easily found to be
$B_k = (1+2Q_\epsilon/K_\epsilon)^{k-1} 2Q_\epsilon/K_\epsilon$ where
$Q_\epsilon$, $K_\epsilon$ and $\epsilon>0$ are given by \eqref{eq:QQQ}, \eqref{eq:KKK} and \eqref{eq:eee}.
Combining these we now obtain
\begin{align*}
\llangle f, g \rrangle_{s+1}
&\ge
\frac{1}{2} K_\epsilon \Big(
\| \nabla f_{0,s} \|_{L^2}^2
+
\sum_{k=1}^{s-1} B_k \| \nabla f_{k,s-k} \|_{L^2}^2
+
\| \nabla f_{s,0} \|_{L^2}^2
\Big)
-
C_{\eta,B}^s \| f \|_{s+1} \| f \|_{\dot{H}^s}
\\
&=
\frac{1}{2} K_\epsilon \| f \|_{s+1}^2
-
C_{\eta,B}^s \| f \|_{s+1} \| f \|_{\dot{H}^s}
\end{align*}
which is the desired estimate.
\end{proof}

\hskip 0.5cm
We can now address the Claim \ref{eq:OMEGA}.
\begin{proposition} \label{omegainvertible}
Let $M=S^1\times [0,\infty)$ and let $\eta(t)$ be a smooth curve of area-preserving diffeomorphisms
$\mathscr{D}_\mu(M)$. Given any $t > 0$
the operator $\widehat{\Omega}_{t} = \int_0^{t} \Lambda_\tau^{-1} \, d\tau$
defined in \eqref{eq:OMEGA} on the space $\mathcal{F}^{s+1}(M)$ of stream functions to itself is invertible.
\end{proposition}
\begin{proof}
For any $0 \leq \tau \leq t$ applying Proposition \ref{positivityprop} to $f= \Lambda_t^{-1} g$ we obtain
\begin{align} \label{eq:ooo}
\llangle g, \Lambda_\tau^{-1} g \rrangle_{s+1}
&\ge
K \lVert \Lambda_\tau^{-1} g \rVert^2_{s+1}
-
C \lVert \Lambda_\tau^{-1} g \rVert_{s+1} \lVert \Lambda_\tau^{-1} g \rVert_{\dot{H}^s}
\\ \nonumber
&\ge
K N_1^{-2} \lVert g \rVert^2_{s+1}
-
C N_2^{-1} N_3^{-1} \lVert g \rVert_{s+1} \lVert g \rVert_{\dot{H}^s},
\end{align}
where
\begin{align*}
N_1 = \sup\limits_{0\leq \tau \leq t} \lVert \Lambda_t \rVert_{s+1},
\quad
N_2 = \inf\limits_{0\leq \tau \leq t} \lVert \Lambda_t \rVert_{s+1}
\quad \text{and} \quad
N_3 = \inf\limits_{0 \leq \tau \leq t} \lVert \Lambda_t \rVert_{\dot{H}^s}.
\end{align*}
Integrating both sides of \eqref{eq:ooo} over $[0, t]$ and using Cauchy-Schwarz we get
$$
\lVert \widehat{\Omega}_t g \rVert_{s+1}
\ge
K t N_1^{-2} \lVert g \rVert_{s+1} - C t N_2^{-1} N_3^{-1} \lVert g \rVert_{\dot{H}^s}.
$$

\hskip 0.5cm
It follows that $\widehat{\Omega}_t$ has closed range.
By Lemma 1, $\widehat{\Omega}_t$ has trivial null-space and
it follows that $\widehat{\Omega}_t$ is semi-Fredholm.
Since the index of semi-Fredholm operators is constant under continuous perturbations,
and since it is zero at $t=0$, we conclude that the index is always zero.
Therefore, $\widehat{\Omega}_t$ also has trivial cokernel and must be invertible on the space of $\mathcal{F}^{s+1}$.
\end{proof}

\hskip 0.5cm
It now follows that given any smooth divergence free vector field $v_0$ on $M$ the corresponding operator
$\Omega_t$ on $T_e\mathscr{D}_\mu$ is also invertible which, in light of Proposition \ref{prop:OLD},
implies that $\Phi_t = D\eta_t (\Omega_t - \Gamma_t)$ is the sum of an invertible operator and a compact operator. We conclude that $\Phi_{t}$ is a Fredholm operator of index zero.
This concludes the proof of Theorem \ref{thm:Fred-b} in the smooth case $v_0 \in T_e\mathscr{D}_\mu$.
The $H^s$ case follows by a perturbation argument as in \cite{EMP06} or \cite{MP10}
and will be omitted. The only important thing to note is that our leading-term estimates depend only on
the $C^1$ norm $\lVert \eta\rVert_{C^1}$, and thus when we approximate an $\eta\in H^s$ by an $\tilde{\eta}\in C^{\infty}$,
the coefficients in the leading term can be made as close as we want to those we found above.


\end{document}